\documentclass[12pt,letterpaper,reqno]{amsart}

\usepackage{times}
\usepackage[T1]{fontenc}
\usepackage{latexsym}

\usepackage{graphicx}

\usepackage{epsfig}
\usepackage{amsmath,amsfonts,amsthm,amssymb,amscd}
\input amssym.def
\input amssym.tex

\usepackage{color}

\renewcommand{\d}{{\mathrm{d}}} 

\newcommand\be{\begin{equation}}
\newcommand\ee{\end{equation}}
\newcommand\bea{\begin{eqnarray}}
\newcommand\eea{\end{eqnarray}}
\newcommand\bi{\begin{itemize}}
\newcommand\ei{\end{itemize}}
\newcommand\ben{\begin{enumerate}}
\newcommand\een{\end{enumerate}}
\newcommand\bc{\begin{center}}
\newcommand\ec{\end{center}}
\newcommand\ba{\begin{array}}
\newcommand\ea{\end{array}}

\newcommand{\C}{\ensuremath{\mathbb{C}}}
\newcommand{\Z}{\ensuremath{\mathbb{Z}}}

\newcommand{\foh}{\frac{1}{2}}  

\newtheorem{thm}{Theorem}[section]

\newtheorem{lem}[thm]{Lemma}

\theoremstyle{definition}

\newcommand{\tht}{\theta}

\newcommand{\cst}{\alpha_0}

\numberwithin{equation}{section}

\begin{document}

\title{The Integral of the  Riemann $\xi$-Function}

\author{Jeffrey C. Lagarias}
\author{David Montague}

\date{August 18, 2011}
\thanks{Work of the authors was supported by NSF grant DMS-0801029}

\begin{abstract}
{This paper studies the
integral of the Riemann $\xi$-function defined by $\xi^{(-1)}(s)= \int_{1/2}^{s} \xi(w) dw$.
More generally,  it studies a one-parameter family of functions given
by Fourier integrals  and satisfying a functional equation.
Members of this family   are shown to
have only finitely many zeros on the critical line, with $\xi^{(-1)}(s)$ 
having exactly one  zero on the critical line, at $s=\frac{1}{2}.$
It is also shown there are zeros  of $\xi^{(-1)}(s)$ that  lie arbitrarily far away from the critical line.
An analogue  of the de-Bruijn-Newman constant is introduced for  this family, 
 and shown to be infinite.}
\end{abstract}
\maketitle

\setcounter{equation}{0}

\section{Introduction}

The Riemann $\xi$-function is the entire function defined by the  formula
 \begin{equation}\label{eq100}
 \xi(s) = \foh s(s-1)\pi^{-\frac{s}{2}}\Gamma(\frac{s}{2})\zeta(s).
  \end{equation}
The $\xi$-function satisfies the functional equation $\xi(s) = \xi(1-s)$, and 
 its zeros are exactly the non-trivial zeros of the Riemann
zeta function $\zeta(s)$, those that lie in the critical strip $0< Re(s) <1$.
The rescaled function  $\Xi(z) := \xi(\frac{1}{2} + iz) =\xi(\frac{1}{2}- iz)$,
obtained
using the variable change $s= \frac{1}{2} + iz$
which sends the critical line 
$Re(s) =\frac{1}{2}$ to the real $z$-axis,
 has the Fourier integral representation
\begin{equation}\label{eq101}
\Xi(z)= 2\int_{0}^{\infty} \Phi(u) \cos zu \, du,
\end{equation}
in which
$$
\Phi(u):=\sum_{n=1}^{\infty} (4 \pi^2 n^4 e^{\frac{9}{2}u} - 6\pi n^2 e^{\frac{5}{2}u}) \
\exp( -\pi n^2 e^{2u}),
~~0 < u < \infty,
$$
is a rapidly decreasing function.
Here $\Xi(z)$ is the original function introduced by Riemann,
see Edwards \cite[p. 18]{Ed74}.

We define the integral of the $\xi$-function to be
\begin{equation}\label{eq102}
\xi^{(-1)}(s) := \int_{\frac{1}{2}}^s \xi(z) dz.
\end{equation}
It  satisfies the functional equation 
\begin{equation}\label{eq103}
\xi^{(-1)}(s)= - \xi^{(-1)}(1-s).
\end{equation}
The rescaled function\footnote{The factor of $-i$ is included here since
$ \frac{d}{ds}=-i \,\frac{d}{dz}$, to make $\Xi^{(-1)}(z)$  real-valued on the real axis.}
 $\Xi^{(-1)}(z) := -i\,\xi^{(-1)}(\frac{1}{2} + iz)$ has the Fourier integral representation
\begin{equation}\label{eq104}
\Xi^{(-1)}(z) = 2\int_{0}^{\infty} \Phi(u) (\frac{\sin zu}{u}) \, du.
\end{equation}
This paper studies the locations of zeros of  this function, and of 
other entire functions related to $\xi^{(-1)}(s),$ defined  below. 

It is a pleasure to dedicate this paper to Akio Fujii, in view
of his long-standing interest in location of the zeros of the Riemann zeta
function (e.g. \cite{Fuj02}, \cite{Fuj03}, \cite{Fuj04}).

\subsection{Background}

There have been many studies of properties of the Riemann $\xi$-function.
This function  motivated the study of functions in the Laguerre-P\'{o}lya class (see P\'{o}lya \cite{Pol27b},
Levin \cite[Chap. 8]{Lev80}),
to which the function $\Xi(z)$ would belong if the Riemann hypothesis
were true. It motivated the study of 
properties of entire  functions represented by Fourier integrals
that are real and bounded on the real axis (see  P\'{o}lya \cite {Pol26a}, \cite{Pol26b}, \cite{Pol27a},
Titchmarsh \cite[Chap. X]{Tit86}, Cardon \cite{Car02})
and  related Fourier transforms (Wintner \cite[Theorems III, IV]{Win36}).
 It  led to the study of 
the effect of various operations on
entire functions, including differential operators and convolution integral operators,  preserving the property
of having zeros on a line (e.g. Craven, Csordas and Smith \cite{CCS87a}, \cite{CCS87},
Craven and Csordas \cite{CC94},
Cardon and Nielsen \cite{CN02}, Cardon and de Gaston \cite{CG05}).
Various necessary conditions for the $\Xi$-function to  to have real
zeros have been verified (Csordas, Norfolk and Varga \cite{CNV86}, Csordas and Varga \cite{CV88}).

In 1976 Newman \cite{New76} introduced
a one-parameter  family of Fourier cosine integrals, given
for real $\lambda$ by
\begin{equation}\label{eq111}
\Xi_{\lambda}(z) := 2\int_{0}^{\infty} e^{\lambda u^2}\Phi(u) \cos zu \, du.
\end{equation}
Here $\Xi_{0}(z) = \Xi(z)$, so this family of functions can be viewed
as deformations of the $\Xi$-function.
 It follows from a 1950 result of de Bruijn \cite[Theorem 13]{deB50} that  
 the entire function $\Xi_{\lambda}(z)$ has only real zeros
for $\lambda \ge \frac{1}{8}.$
Newman \cite{New76} proved that there
exists  a real number  $\lambda_0$ such that\footnote{Our
definition of $\Xi_{\lambda}(t)$ corresponds to the  function $\Xi_{b}(t)$ with $b= -\lambda$
in Newman's paper.}
$\Xi_{\lambda}(t)$ 
has all real zeros for $\lambda \ge \lambda_0$,
and has some nonreal zeros for each $\lambda < \lambda_0$.
The Riemann hypothesis holds if and only if $\lambda_0 \le 0$,
and Newman conjectured that the converse inequality $\lambda_0 \ge 0$ holds.
Newman \cite[Remark 2]{New76} stated that his  conjecture represents a quantitative 
version of the assertion that the Riemann hypothesis, if true, is just
barely true.
The rescaled value $\Lambda := 4\lambda_0$ was later 
named by Csordas, Norfolk and Varga \cite{CNV88}
the {\em de Bruijn-Newman constant,}  
and  they proved that $-50 \le \Lambda $.
Successive authors obtained better bounds obtaining  by finding two zeros of
the Riemann zeta function that were unusually close together.
Successive improvements of examples on close zeta zeros  led to the lower bound
$$
-2.7 \times 10^{-9} < \Lambda.
$$
obtained by  Odlyzko \cite{Odl00}. Recently Ki, Kim  and Lee \cite[Theorem 1]{KKL09}
established that $\Lambda < \frac{1}{2}.$
The conjecture that $\Lambda=0$ is now termed
the de Bruijn-Newman conjecture.
Odlyzko \cite[Sect. 5]{Odl00}  observed that the existence of very close spacings of zeta zeros,
would imply the truth of the
de Bruijn-Newman conjecture.

In another direction, one may  consider the effects of  differentiation on 
the location and spacing of zeros of an entire function $F(z)$.
In  1943 P\'{o}lya \cite[p. 182]{Pol43} conjectured that an entire function $F(z)$ of order less than $2$ that has
only a finite number of zeros off the real axis, has the property that there exists a finite 
$m_0 \ge0$ such that all successive derivatives $F^{(m)}(z)$ for $m \ge m_0$ have only real zeros.
This was proved by Craven, Csordas and Smith \cite{CCS87} in 1987, with a
new proof given by Ki and Kim \cite{KK00} in 2000.
Farmer and Rhoades \cite{FR05} have shown 
(under certain hypotheses) that  differentiation of an entire  function
with only real zeros will yield a function having real zeros whose 
zero distribution on the real line 
is ``smoothed.''   Their results apply to the Riemann $\xi$-function,
and imply that if the Riemann hypothesis holds, then the same will
be true for all derivatives $\xi^{(m)}(s) = \frac{d^m}{ds^m} \xi(s)$, $m \ge 1$.
Various general results are given in Cardon and  de Gaston \cite{CG05}.

Passing to results on derivatives of the  $\xi$ -function, 
in 1983 Conrey \cite{Con83} unconditionally showed 
that the $m$-th derivative $\xi^{(m)}(s)$ of the $\xi$-function necessarily has 
a positive fraction of its zeros falling
on the critical line, and his lower
bound for this  fraction increases towards $1$ as $m$ increases.
In 2006 Ki \cite{Ki06}  proved a conjecture of Farmer and Rhoades,
showing that there exist positive sequences $A_m, C_m$, with $C_m \to 0$
slowly with $m$, such that
$$
\lim_{m \to \infty} A_m \, \Xi^{(2m)}(C_m z) = \cos z.
$$
This result can be viewed as  quantitative version of the assertion  that for the $\xi$-function  differentiation
smooths out the spacings of the zeros, since
$\cos z$ has perfectly spaced zeros.  (See Coffey \cite{Cof09} for a related result.)

In 2009 Ki, Kim and Lee \cite{KKL09} combined differentiation with the de Bruijn-Newman constant. 
For each integer $m \ge 0$ they introduced the family of functions
$$\Xi^{(m)}_{\lambda}(z) :=\frac{d^m}{dz^m} \Xi_{\lambda}(z),$$
 depending on the real parameter $\lambda$. These are given by the Fourier integrals
 \begin{equation}\label{eq117}
 \Xi^{(m)}_{\lambda}(z) 
=\left\{ \begin{array}{ll} \int_{0}^{\infty} e^{\lambda u^2} u^{2n}~\Phi(u) ((-1)^n {\cos {zu}} )~ du 
& \mbox{for}~m= 2n, ~~~~~n \ge 0,\\
& ~\\
\int_{0}^{\infty} e^{\lambda u^2} u^{2n-1}\Phi(u) ((-1)^n {\sin {zu}})  ~du 
& \mbox{for}~m=2n-1,~ n \ge 1.\end{array}  \right.
\end{equation}
To each of these families  they associated a de Bruijn-Newman-like  constant,
first defining
\begin{equation}\label{eq118}
\lambda_m := \inf \{ \lambda: \Xi^{(m)}_{\lambda}(z)~\mbox{has all zeros real} \},
\end{equation}
and then setting $\Lambda^{(m)} := 4 \lambda_m$.
The case  $\Lambda^{(0)}= \Lambda$ recovers the original de Bruijn-Newman constant.
They proved that 
$$
\Lambda^{(0)} \ge \Lambda^{(1)} \ge \Lambda^{(2)} \ge \cdots,
$$
and that
$$
\lim_{m \to \infty} \Lambda^{(m)} \le 0.
$$

Finally we remark that $\xi(s)$ is an even function around
the point $s= \frac{1}{2}$, having there a
Taylor series expansion 
$$
\xi(s) = \sum_{j=0}^{\infty} \frac{c_{2j}}{(2j)!}\, (s- \frac{1}{2})^{2j},
$$
with coefficients $c_{2j}=\xi^{(2j)}(\frac{1}{2})  $ that are real and positive.  
The maximum modulus $M(r) := \max\{ |\xi(\frac{1}{2} +iz)|: |z| = r\}$ is
therefore attained for $i z$ on the real axis. In 1945 Haviland \cite{Hav45} obtained
an asymptotic expansion for $M(r)$ of the shape
$$
M(r) \sim (\frac{1}{2} \pi)^{\frac{1}{4}}(2\pi e)^{-\frac{1}{2} r} r^{\frac{1}{2} r + \frac{7}{4}}
\Big( \sum_{n=0}^{\infty} \frac{C_n}{r^n} \Big),
$$
having $C_0=1.$
From the integral  (\ref{eq102}) we deduce the Taylor expansion 
$$
\xi^{(-1)}(s) = \sum_{j=0}^{\infty} \frac{c_{2j}}{(2j+1)!}\, (s- \frac{1}{2})^{2j+1},
$$
manifestly showing that  $\xi^{(-1)}(s)$ is  an odd function around $s= \frac{1}{2}.$
Coffey (\cite{Cof04}, \cite{Cof09})  found integral formulas for the coefficients
$c_{2j}$ and determined their asymptotics as $j \to \infty$.

%
%

\subsection{Present Work}
To add perspective to the results above,  we
study  the effect of the  inverse operation of integration applied to 
the Riemann $\xi$-function  on the zeros  of 
the resulting function. 
Since differentiation seems to smooth the
distribution of zero spacings, we may anticipate that integration will ``roughen''
their distribution, and even force zeros off the critical line.  Our
object is to obtain quantitative information in this direction. 
We study several variants of  the function $\xi^{(-1)}(s)$,
including a family of  functions defined  in analogy with  $\Xi_{\lambda}(z)$ above.

Based on the Fourier integral representation (\ref{eq104}), we
define an  analogue for $m=-1$
of the one-parameter  families  of
functions $\Xi_{\lambda}^{(m)}$ studied by Ki et al. \cite{KKL09}, as follows.
Given a real $\lambda$,  set 
\begin{equation}\label{eq122}
\Xi^{(-1)}_{\lambda}(z) := 2\int_{0}^{\infty}e^{\lambda u^2} \Phi(u) \left( \frac{\sin zu}{u}\right) du.
\end{equation}
The functions $\Xi^{(-1)}_\lambda(z)$  are odd functions, are real on the real axis, and they
satisfy $\frac{d}{dz} \Xi^{(-1)}_{\lambda}(z) = \Xi_{\lambda}(z).$  
For this family we may define
 a de Bruijn-Newman constant  for $m= -1$, 
by analogy with the definition above: we  first set
\begin{equation}\label{eq124}
\lambda_{-1} = \inf \{ \lambda:~ \Xi^{(-1)}_{\lambda}(z) ~~\mbox{has all zeros real}\}
\end{equation}
and then set
$\Lambda^{(-1)} := 4 \lambda_{-1}.$  
In the paper we will show
 that $\Lambda^{(-1)} = +\infty$.
That is,   we show that for each real  $\lambda$ the function
$\Xi_{\lambda}^{(-1)}(z)$ has at least one non-real zero; in fact, it has
infinitely many non-real zeros.

In another direction, concering the function $\xi^{(-1)}(s)$ defined by (\ref{eq102}), we
introduce a constant of integration $\cst \in \C$, and  define
\begin{equation}\label{eq126}
\xi^{(-1)}(s; \cst) : =    \xi^{(-1)}(s)+\cst  =\int_{\frac{1}{2}}^s \xi(w)dw + \cst.
\end{equation}
The functional equation for $\xi^{(-1)}(s)$ then yields
$$
\xi^{(-1)}(s; \cst) = -\xi^{(-1)}(1-s, -\cst).
$$
The problem of determining  the zero set  of $\xi^{(-1)}(s; -\cst)$ with 
integration constant $-\cst$ is the same as that of detemining
the  set of points where $\xi^{(-1)}(s)= \cst$, which we call
the $\cst$-value set of $\xi^{(-1)}$,
and denote $V(\xi^{(-1)}; \cst).$   We  obtain detailed information on
the value sets, showing that for all but two values of $\cst$ only
finitely many zeros are on the critical line, and 
for all values of $\cst$ there are zeros arbitrarily
far off the critical line.

We give precise statements of results in Section 2; we then 
 discuss consequences of these results. Sections 3 to 5 give proofs.
In Section 3 we collect preliminary results needed for  proofs of these
results
 In Sections 4 and 5 we give proofs of  the main theorems.
In the final Section 6  we present numerical results on zeros of $\xi^{(-1)}(s)$
and related functions, and raise some open questions. \medskip


 \paragraph{\bf Acknowledgments.}
 We thank Henri Cohen for useful discussions and computations
 reported in Section \ref{sec6}, and
 Jon Bober for help with plots of this data. 
 We thank the reviewer for 
   many detailed and helpful corrections, motivating  a 
   substantive revision
   of the proof of Theorem ~\ref{th22}. We thank
  P\"{a}r Kurlburg for  useful conversations, and  Steven Finch
  for noting some  misprints. 
   Some work of the first author was done while visiting MSRI,
      as part of the Arithmetic Statistics  Program. MSRI  is
    supported by the National Science Foundation. \\

%

\section{Results}

We study   the set of real zeros of functions in the family 
$\Xi_{\lambda}^{(-1)}(z)$, and  determine
information on   the zero sets of $\xi^{(-1)} (s) - \cst$
for arbitrary values $\cst \in \C$. We obtain  two main results.

%

\subsection{Behavior of\ \ $\Xi_{\lambda}^{(-1)}(z)$}

The first result concerns the behavior of the function $\Xi_{\lambda}^{(-1)}(z)$
defined in (\ref{eq122}) on the real axis.

%
%
\begin{thm}\label{th21}
For real $\lambda$, 
the functions 
 $\Xi_{\lambda}^{(-1)}(z)$
 have the following properties.
\begin{enumerate}
\item
For each real $\lambda$, one has 
$$
\lim_{ t \to \infty} \Xi_{\lambda}^{(-1)}( t) = A_0,
$$
where $A_0$ is a nonzero constant independent of $\lambda$ given by
$$
A_0 := \pi \Phi(0) \approx 2.80668.
$$
The value $A_0=  \frac{\pi}{2}\left(4\tht^{''}(1) + 6\tht^{'}(1)\right)$, taking
$\theta(z)= \sum_{n \in \mathbb{Z}} e^{-\pi n^2 z}$.

\item
For each real $\lambda$, the function $\Xi_{\lambda}^{(-1)}(t)$
has finitely many  zeros on the real $t$-axis, and has
infinitely many non-real zeros.  The zeros on the real axis always include
 a zero at $t=0$, and for $\lambda \le 0$ this 
is the only real zero of $\Xi_{\lambda}^{(-1)}(t)$.
\end{enumerate}
\end{thm}

Since $\Xi_{\lambda}^{(-1)}(z)$ is an odd function, we have
$$
\lim_{ t \to -\infty} \Xi_{\lambda}^{(-1)}( t) = -A_0.
$$
This  theorem shows that the ``just barely true''
heuristic for the Riemann hypothesis holds in a particularly
strong fashion for the operation of integration.
Namely, integration drives all but finitely many
zeros off the real axis for
every function in  the  family $\Xi_{\lambda}^{(-1)}(z)$.

The result  (1) above  is derived 
directly from the oscillatory integral representation for $\Xi_{\lambda}^{(-1)}(t)$;
it corresponds to a Fourier  sine integral (on a half line)
against a function having a singularity at the endpoint $u=0$.
The proof of Theorem \ref{th21} obtains the bound, valid for $t \ge 3$,
$$
\Xi_{\lambda}^{(-1)}(t) = A_0 + O \left ( t^{-2/3}\right),
$$
in which the implied constant  in the error term depends on $\lambda$.

The result (2) above shows that the analogue of the de Bruijn-Newman constant
for this integral of the $\xi$-function  fails to exist; that is,  it establishes
$$
\Lambda^{(-1)}= +\infty.
$$
In  1947 Wintner \cite{Wi47} proved  that 
$\Xi_{0}^{(-1)}(t) >0$ when $t>0$; this fact together with
the functional equation (\ref{eq103})
 implies that $\xi^{(-1)}(s)$ has
no zeros on the critical line except for a zero at $s=\frac{1}{2}.$
Wintner's approach extends to cover the case   $\lambda \le 0$, as stated in (2) above.

%

\subsection{Value Distribution of $\xi^{(-1)}(s)$}

The  second result  concerns the location of  zeros of 
the function $\xi^{(-1)}(s)$ given in (\ref{eq102}).
More generally it studies the distribution of  values $\xi^{(-1)}(s)= \alpha$ of this function.
For an entire function $f(z)$, let $Z(f)$ denote the set
of zeros of $f$, where we count zeros with multiplicity; thus
$Z(f)$ is a multiset.
We define the {\em value set} of $f$ at value $\alpha$ by
$$
V(f ;\alpha) := Z( f- \alpha)=  \{ z \in \C: ~~f(z)-\alpha=0\}.
$$
Theorem~\ref{th21} gives two exceptional limiting values
$$
\lim_{t \to \pm \infty} \xi^{(-1)} (\frac{1}{2} + it) = \pm iA_0.
$$
We show that, aside from these two values $\alpha = \pm i A_0$, 
the locations of all values in $V(\xi^{(-1)}; \alpha)$ are
qualitatively similar in  
asymptotics relating  the real part of zeros to their
imaginary part. In what follows $4\pi e \approx 34.1588$ serves as
as a useful cutoff value.

%
%
\begin{thm}\label{th22}
For the function $\xi^{(-1)}(s)$ 
and  $\cst \in \C$, consider  the set of points  
where $\xi^{(-1)}(s)= \cst.$
All such points  $\rho=\sigma+ it$ having $|t| \le 4\pi e$ lie in a bounded region,
which depends on $\cst$.

(1) For each $\cst \ne \pm i A_0$, 
 all members $\rho$
 having  $|t| \ge 4\pi e$ satisfy
\begin{equation}\label{eq222}
|\sigma| = \frac{\pi}{2} \frac{ |t|}{\log |t|} + O\left(  \frac{|t|}{(\log |t|)^2} \right),
\end{equation}
in which the implied constant in the $O$-symbol depends on $|\cst|$.
For $\cst=  i A_0$ (resp. $\cst= -i A_0$)  this bound applies
when  $t \le -4\pi e$
(resp. $t \ge 4\pi e$).

(2) For $\cst=  i A_0$ (resp. $\cst= -i A_0$)  the upper bound applies
when $t\ge 4\pi e$ (resp. $t \le -4\pi e$) :

\begin{equation} \label{eq223}
|\sigma| \le  \frac{\pi}{2} \frac{ |t|}{\log |t|} + O\left(  \frac{|t|}{(\log |t|)^2} \right),
\end{equation}
\end{thm}

This result shows that the value distribution
  of $\xi^{(-1)}(s)$ is qualitatively the same for all values $\cst \ne \pm iA_0$,
as well as for the values $\cst = iA_0$  in the lower half plane, and $\cst = -iA_0$
 in the upper half-plane.
 The two remaining cases in (2) appear to  have
a different distribution; numerical evidence given in \S6 supports
this possibility.
 In the exceptional cases in (2) we suspect that 
 the values remain closer to the critical line, at least to the
 extent that 
$|\sigma| = O (\frac{|t|}{(\log |t|)^2})$
might hold in these two cases. 

Since the functions $\xi^{(-1)}(s)- \cst$ have infinitely
many zeros (because they are entire functions of order $1$
of maximal type, cf. Lemma \ref{le42}), we deduce
that all of them have
zeros arbitrarily far away from the critical line.
(In the case $\cst =  iA_0,$ resp. $-iA_0$,  one must additionally 
show that they have
infinitely many zeros with negative real part, resp. positive real part.)

\subsection{Discussion}

First, Theorem ~\ref{th21} is obtained by viewing
the functions $\Xi^{(-1)}_{\lambda}(t)$ as 
Fourier integrals of functions $\Xi^{(-1)}(z) = \frac{1}{i}\int_{-\infty}^{\infty} \Psi(u) e^{izu} du,$
where the function $\Psi(u)= \frac{1}{u} \Phi(u)$ has a singularity at the 
point $u=0$. Since $\Psi(u)$
is an odd function, this integral can be rewritten as an absolutely convergent integral
 $2 \int_{0}^{\infty} \Phi(u) (\frac{\sin zu}{u}) du.$ 
 The singularity at $u=0$, which occurs since $\Phi(0) \ne 0$, results in
a nonzero integral on the critical line, and this  is the mechanism that
forces zeros off the real axis. In this regard, one may consider 
more generally an oscillatory integral 
$F(z)= \int_{0}^{\infty} \Psi(u) \cos(zu) du$ 
 in which
$\Psi(u)$ is a smooth function with very rapid decay as $u \to \infty$, such that $F(z)$ is an
entire function. Then the order $k$ of  the zero at $x=0$ of a smooth function
$\Psi(u)$  in an oscillatory integral  places an absolute  limit on 
the number of  integrations  $F^{(-j)}(z)$ of $F(z)$ that can be taken
(with any choices of constant of integration) 
to have the property that  $F^{(-j)}(t) \to 0$ holds as 
 $t \to \pm \infty$
for each $1 \le j \le m$; it requires that  $m \le k$.

Second, one may ask how the zeros on the
real axis of $\Xi_{\lambda}^{(-1)}(t)$ behave
as $\lambda \to \infty$. Theorem \ref{th21} implies there are
a finite number of zeros for each $\lambda$. It may be that
this number increases as $\lambda \to \infty$, and that new zeros
are created in pairs at the origin at certain values
of $\lambda$ as it increases, and afterwards have
 a regular behavior as a function of $\lambda.$

Third, one may ask whether the constant $A_0 = 
\frac{\pi}{2}\left(4\tht^{''}(1) + 6\tht^{'}(1)\right)$ appearing
in Theorem~\ref{th21} may possibly have an  arithmetic
interpretation. It is known that the  value $\tht(1)= \frac{\pi^{1/4}}{\Gamma(\frac{3}{4})}$ has 
an arithmetic interpretation in the context of the Arakelov zeta function
studied in Lagarias and Rains \cite[Appendix]{LR03}.

Fourth, in connection with the smoothing property on the zeros distribution  of
taking derivatives of $\xi(s)$, one may inquire concerning the level spacing
distribution of $\xi^{(n)}(s)$ for $n \ge 1$. The GUE conjecture 
asserts that the level spacing distributions of 
 $k$ consecutive normalized zeros of $\xi(s)$ has a limiting distribution specified by
 the GUE distribution.  It seems 
plausible to expect that $\xi^{(n)}(s)$ will have its own level spacing distribution
${\rm GUE}^{(n)}$ which will differ from that of the GUE. It would be interesting
to make a prediction for  ${\rm GUE}^{(n)}$, and if possible,  to find a random matrix model
for it.

Fifth, Theorem \ref{th22} 
gives a functional bound relating  the horizontal and vertical coordinates of individual values.
It remains to determine the asymptotics of the
 vertical distribution of the zeros of $\xi^{(-1)}(s)$, or more
 generally of any fixed value set $\xi^{(-1)}(s) = \cst.$ We might expect these values
to   obey approximately the same asymptotics as that
of the Riemann $\xi$-function, which is 
(\cite[Chap. 15]{Dav80})
$$
N(T) =\frac{1}{\pi} T \log T - \frac{1}{2\pi} \log (\frac{1}{2\pi e}) T + O(\log T),
$$
as least in the main term $\frac{T}{\pi}\log T$ in the asymptotics.
For $\xi^{(-1)}(s)$ we note 
 the weak result $N(T) = O( T\log T)$  follows
from a Jensen's formula estimate from growth of its maximum modulus.

%
\section{Basic Observations}

The  $\xi$-function can be expressed  in terms of a Mellin
 transform of derivatives of the Jacobi theta function 
 $ \vartheta_3(0, q) = \sum_{n \in \Z} q^{n^2} .$
 We make the  variable change $q=e^{-\pi z}$,  and
 on the half-plane  $Re(z)>0$
 define the {\em theta function}
 \[ \tht(z) := \sum_{n=-\infty}^\infty e^{-\pi n^2 z} \]
 We write $z= x+iy$, and 
 mainly consider $\tht(x)$ restricted to the real axis. 
For the derivatives of $\tht$, we have the formulas
\[ 
\tht'(x) = \sum_{n=-\infty}^\infty -\pi n^2 e^{-\pi n^2 x} \ \ \ \ \ \ \ \ \ 
\tht''(x) = \sum_{n=-\infty}^\infty \pi^2 n^4 e^{-\pi n^2 x}. 
\]
The function $\Phi(u)$ given in Sect. 1 is expressible in terms of 
derviatives of the theta function as given in (\ref{eq423}) below.

%
\begin{lem}\label{le41}
The function $\xi(s)$ is given by the Fourier cosine transform
\begin{equation}\label{eq420}
\Xi(z) = \xi(\frac{1}{2} + iz) = 2 \int_{0}^{\infty} \Phi(u) \cos zu \, du
\end{equation}
in which
\begin{equation}\label{eq421}
\Phi(u) = \sum_{n=1}^{\infty} \left(4\pi^2 n^4 e^{\frac{9}{2} u} - 6 \pi n^2 e^{\frac{5}{2}u}\right)
\exp \left( - \pi n^2 e^{2u}\right).
\end{equation}
The function $\Phi(u)$ has the following properties.
\begin{enumerate}
\item
$\Phi(u) = \foh \left( \frac{d^2}{du^2} - \frac{1}{4}\right) \left( e^{\frac{1}{2}u} \tht(e^{2u}) \right).$
\item
$\Phi(u)$ is an even function: $\Phi(u) = \Phi(-u)$.
\item
$\Phi(u)$ decays extremely rapidly on the real axis as $u \to \pm \infty$, with
$$\Phi (u)\ll \exp \left( -e^{|u|} \right)~~\mbox{as}~u \to \pm \infty.$$
\item
$\Phi(u)$ analytically continues to the strip $| \mbox{Im}(u) | < \frac{\pi}{4}$.
For integer $m \ge 0$, it satisfies, allowing only real $t$ in the limit, 
$$
\lim_{t \to \pi/4} \Phi^{(m)}(it) = 0.
$$
\item
$\Phi(u)$ is a strictly decreasing function on $[0, \infty)$.
\end{enumerate}
\end{lem}

\begin{proof}
We start from Riemann's formula (\cite{Rie1859})
$$
 \xi(\frac{1}{2}+ it) = 
 4 \int_{1}^{\infty} \frac{d[ x^{3/2} (\psi^{'}(x)]}{dx} x^{-\frac{1}{4}} \cos (\frac{t}{2} \log x) dx,
$$
in which $\psi(x) = \frac{1}{2}( \theta(x)-1)$, so that $\psi'(x) = \frac{1}{2} \theta'(x)$,
and make the variable change
 $x=e^{2u}$ to obtain (\ref{eq420}), cf.  Edwards \cite[Sec. 1.8]{Ed74},
or Titchmarsh \cite[Sec. 10.1]{Tit86}.
 The expansion of the theta function yields
\begin{equation}\label{eq423}
\Phi(u) = [ 3 x^{\frac{5}{4}} \tht^{'}(x) + 2 x^{\frac{9}{4}} \tht^{''}(x)]\mid_{x= e^{2u}}.
\end{equation}
We now consider properties of $\Phi(u).$

(1) This formula was noted by P\'{o}lya \cite{Pol26a} in 1926. It can be directly verified by
 comparison of the right side with (\ref{eq423}).

(2) The functional equation $\tht(x) = \sqrt{\frac{1}{x}} \tht(\frac{1}{x})$
yields $e^{\frac{u}{2}} \tht( e^{2u}) = e^{-\frac{u}{2}} \tht(e^{-2u}).$  Substituting this
in (1) yields $\Phi(u)= \Phi(-u)$.

(3) For $u \to +\infty$ this follows by inspection of (\ref{eq421}). For $u \to -\infty$ it follows using (2).

(4) The theta function $\tht(x)$ defines an analytic function on the right half plane ${\rm Re}(x) >0$.
Under the change of variable $x=e^{2u}$, this region corresponds to the strip $|{\rm Im}(u)| < \frac{\pi}{4}.$
The limiting values were noted by P\'{o}lya \cite{Pol26a}. (Fact (4) is not used in this paper.)

(5) The decreasing property of $\Phi(u)$ was proved in 1935 by Wintner \cite{Wi35}.
Note that $\Phi(0) =\sum_{n=1}^{\infty} (4\pi^2 n^4 -6\pi n^2) e^{-\pi n^2} \approx 0.89339$.
\end{proof}

We next give basic properties of the family of functions $\Xi_{\lambda}^{(-1)}(z)$.
%
\begin{lem}\label{le42} 
For real $\lambda$ the  functions $\Xi^{(-1)}_{\lambda}(z)$  have the  following properties.
\begin{enumerate}
\item
Each function $\Xi^{(-1)}_{\lambda}(z)$
is an entire function of $z$ of order $1$ and maximal type.
\item
Each function  is real on the real axis and is an odd function, i.e.
$$
\Xi_{\lambda}^{(-1)}(z)= - \, \Xi_{\lambda}^{(-1)}(-z).
$$
Thus it is pure imaginary on the
imaginary axis $z=it$.
\item
One has  $\frac{d}{dz} \Xi_{\lambda}^{(-1)}( z) = \Xi_{\lambda}(z).$
Thus, for $\lambda=0$,
$$
\Xi_{0}^{(-1)}(z) = -i \, \xi^{(-1)}(\frac{1}{2} + iz) = -i  \int_{1/2}^{1/2+ iz} \xi(w) dw.
$$
\end{enumerate}
\end{lem}

\begin{proof}
(1) The integral representation  $\Xi^{(-1)}_{\lambda}(z)=2\int_{0}^{\infty} e^{\lambda u^2} \Phi(u) \left( \frac{\sin zu}{u} \right) du$ shows that it is an
entire function of $z$. The results of P\'{o}lya \cite[pp. 9-10]{Pol27a} imply
it is of order $1$.
 Viewing the representation as a Fourier integral,
the  Paley-Wiener theorem shows it cannot have growth of order $1$ and finite type,
whence it has maximal type.

(2) The integral representation  shows that $\Xi_{\lambda}^{(-1)}(z)$ is
real on the real axis, and it is clearly an odd function of $z$. 
Since
$\sin (izu) = -i\sinh(zu)$ we conclude this function is pure imaginary on the imaginary axis.

(3) The rapid decay of $\Phi(u)$ as $|u| \to \infty$ permits differentiation under the integral sign.
Now the identity $\Xi_{0}(z)= \xi(\frac{1}{2} + iz)$  and the fact that 
 $\Xi_{\lambda}^{(-1)}(0)=0$ by (2) yield the last equation.
 (Note that $ds=i \,dz$.)
\end{proof}


Lemma \ref{le42}(2) above implies that the zeros of $\Xi_{\lambda}^{(-1)}(z)$
necessarily have a four-fold symmetry about the
real and imaginary axes: If $\rho$ is a zero of any $\Xi_{\lambda}^{(-1)}(z)$,
then so are $\bar{\rho}$, $-\rho$ and $-\bar{\rho}$. 
Similarly  $\xi_{\lambda}^{(-1)}(s) = i \Xi_{\lambda}^{(-1)}(i(\frac{1}{2}-s))$
necessarily has zeros obeying the same four-fold symmetry as
those of the $\xi$-function: If $\rho$ is a zero of any $\xi_{\lambda}^{(-1)}(s)$,
then so are $\bar{\rho}$, $1-\rho$ and $1-\bar{\rho}$. 
\medskip

To prove Theorem \ref{th22} we will use  estimates
on the size of $\xi(s)$,
 derived using the factorization (\ref{eq100}) for $\xi(s)$. 
To state these it
is convenient to introduce the function $F(\sigma, t)$ of two real variables defined
 for $\sigma \ge 0$, $t >0$ by
 \begin{equation}\label{eq333}
F(\sigma, t) :=
\sqrt{\pi}
(2\pi e)^{-\frac{\sigma}{2} }\ (\sigma^2 +t^2)^{\frac{\sigma +3}{4}} \exp \left(-\frac{t}{2} \arctan (\frac{t}{\sigma} ) \right).
\end{equation}
We have the following estimates.


\begin{lem}\label{le43}
There are positive constants $C_{1}, C_{2}, C_{3}$ with the following properties.

(1) For $\frac{1}{2} \le Re(s) \le 2$, the
 function $\xi(s)$ satisfies
\begin{equation}\label{eq331}
|\xi(s)| \le C_{1}e^{-\frac{\pi}{4} |t|} (|t| +1)^{5/2}.
\end{equation}

(2)  For  $s = \sigma+ it$ with $\sigma \ge 2$ and
all real $t$, there holds
\begin{equation}\label{eq332}
F(\sigma, t)\left(1 -  C_{2}(\frac{1}{|\sigma+it|}+ 2^{-\sigma}) \right) \le |\xi(s)| \le
 F(\sigma, t)\left(1 + C_{3}(\frac{1}{|\sigma+it|} + 2^{-\sigma}) \right),
\end{equation}
with $F(\sigma, t)$ given by (\ref{eq333}).

(3) For each $\delta > 0$, there is a positive constant $C = C(\delta)$ such that for 
all $\sigma_0 > C$ and all real $t$, the function $\xi(\sigma + it)$ satisfies
$$
\int_{\sigma_0}^{\sigma_0+2} |\frac{d \arg \xi(\sigma + it)}{d\sigma} |d \sigma \leq 
\frac{1+\pi}{2} + \delta.
$$

\end{lem}

\begin{proof}

(1)
 For $\frac{1}{2} \le Re(s) \le 2$, and $|t| \ge 2$, we have the estimates
$$
|\frac{1}{2}s(s-1)| = O(|t|^2),
$$ 
$$
|\pi^{-s/2}| = O(1),
$$
and
$$
|\Gamma(s/2)| = O(e^{-\frac{\pi |t|}{4}}).
$$
The convexity bound $|\zeta(s)| \leq C(\epsilon) |t|^{\frac{1}{2} - \frac{1}{2} \sigma+\epsilon}$
 (\cite[Chap. V]{Tit86}), valid uniformly for $|t| \ge 2$,
 yields for all $\sigma \ge \frac{1}{2}$ and $|t| \ge 2$,
$$
|\zeta(s)| \leq C |t|^{\frac{1}{2}}.
$$
Combining all these estimates, we easily obtain, for $\frac{1}{2} \le \sigma \le 2$
and all real $t$,
$$|\xi(s)| = O(e^{-\pi t/4}(|t|+1)^{\frac{5}{2}}).$$

(2) 
By definition, $\xi(s) = \foh s(s-1) \pi^{-s/2}\Gamma(s/2)\zeta(s)$. 
Suppose $\sigma \ge 2$ and $t$ is arbitrary.
On this domain 
$$
|s(s-1)| = |\sigma+ it|^2(1+ O(\frac{1}{|\sigma+ it|})).
$$
and
$$
|\zeta(s)| = 1+ O (|2^{-s}|) = 1+ O( 2^{-\sigma}) 
$$
and $|\pi^{-\frac{s}{2}}| = e^{-\frac{\sigma}{2} \log \pi}$.
Now Stirling's formula gives, for $Re(s) \ge \frac{1}{2}$, 
\begin{eqnarray*}
|\Gamma(\frac{s}{2})|  &=& \exp \left( Re\left( \frac{s}{2} \log \frac{s}{2} - \frac{s}{2} + 
\frac{1}{2} \log (\frac{ 4 \pi}{s}) 
+ O (\frac{1}{s})\right) \right)\\
&=& \exp \left( \frac{\sigma}{2} \log |\frac{s}{2}| - \frac{t}{2}\arctan(\frac{t}{\sigma})-\frac{\sigma}{2}+\frac{1}{2}\log 4\pi -\frac{1}{2} \log |s| 
+ O (\frac{1}{|s|}) \right)
\end{eqnarray*}
Combining all of the above estimates, we obtain
\begin{eqnarray*}
|\xi(s)| & = & |\foh s(s-1) \pi^{-s/2}\Gamma(s/2)\zeta(s)| \\
& = & \foh |\sigma+ it|^2\left(1+ O(\frac{1}{|\sigma+ it|})\right)\left(1+ O( 2^{-\sigma}) \right)e^{-\frac{\sigma}{2} \log \pi} \cdot \\
&    & \quad \cdot \exp \left( \frac{\sigma}{2} \log |\frac{\sigma+ it}{2}| - \frac{t}{2}\arctan(\frac{t}{\sigma})-\frac{\sigma}{2}+\frac{1}{2}\log 4\pi -\frac{1}{2} \log |s| 
+O (\frac{1}{|s|}) \right) \\ 
& = & \sqrt{\pi} |\sigma + it|^{\frac{3}{2}}e^{-\frac{\sigma}{2} \log (2\pi e)} \exp\left(
\frac{\sigma}{4} \log(|\sigma+it|^2)
- \frac{t}{2}\arctan(\frac{t}{\sigma}) \right)
 \cdot \\
&    ~&
\quad  \cdot \left(1+O(\frac{1}{|\sigma+it|} + \frac{1}{2^\sigma})\right)\\
&=&
F(\sigma, t) \left(1+O(\frac{1}{|\sigma+it|} + \frac{1}{2^\sigma})\right).
\end{eqnarray*}

(3) We derive estimates related to $\arg(\xi(s))$, taken to be $0$ on the real axis for 
$\sigma > 1$.  Then, for $\sigma > 1,$ 
\begin{eqnarray*}
\arg \xi(s) & = & \arg(s) + \arg(s-1) + \arg(\pi^{-s/2}) + \arg(\zeta(s)) + \arg(\Gamma(s/2)) \\
& = & \arctan(\frac{t}{\sigma}) + \arctan(\frac{t}{\sigma-1}) - \frac{t\log \pi}{2} + \arg(\zeta(s)) + \arg(\Gamma(s/2)). 
\end{eqnarray*}

First, note that $$\frac{d \arg \zeta(\sigma+it)}{d\sigma} = \frac{d\ \text{Im}( \log \zeta(\sigma+it))}{d \sigma} = \text{Im} \left(\frac{\zeta'}{\zeta}(\sigma+it)\right),$$ and since we have uniformly in $t$ that
$$\lim_{\sigma \rightarrow \infty} \zeta(\sigma+it) = 1, \ \ \ \ \ \lim_{\sigma \rightarrow \infty} \zeta'(\sigma+it) = 0,$$
we can choose $C$ sufficiently large so that $\sigma_0 > C$ implies that $|\frac{d \arg \zeta(\sigma+it)}{d\sigma}| < \delta/4$. If we also choose $C$ large enough so that $|\arctan(\frac{t}{\sigma_0+2})-\arctan(\frac{t}{\sigma_0- 1})| < \delta/4$, then

\begin{eqnarray*}
\int_{\sigma_0}^{\sigma_0+2} |\frac{d \arg \xi(\sigma+it)}{d\sigma}| d \sigma & \leq &
 2\mid \arctan(\frac{t}{\sigma_0+2})-\arctan(\frac{t}{\sigma_0 -1})\mid \\ 
     & & + \int_{\sigma_0}^{\sigma_0+2} |\frac{d \arg \zeta(\sigma+it)}{d\sigma}| d \sigma 
      + \int_{\sigma_0}^{\sigma_0+2} |\frac{d \arg \Gamma(\frac{\sigma+it}{2})}{d\sigma}| d \sigma \\
   & \le & \delta + \int_{\sigma_0}^{\sigma_0+2} |\frac{d \arg \Gamma(\frac{\sigma+it}{2})}{d\sigma}| d \sigma.
\end{eqnarray*}
By Stirling's formula, we have
$$
\arg(\Gamma(\frac{s}{2})) = \frac{t}{2}\log|\frac{\sigma + it}{2}| + \frac{\sigma}{2}\arctan(\frac{t}{\sigma}) - \frac{t}{2} 
-\frac{1}{2} \arctan(t/\sigma)
+ O(\frac{1}{|\sigma+it|}).
$$
Therefore, for $\sigma >1$, 
\bea
\frac{d \arg \Gamma(\frac{\sigma + it}{2})}{d\sigma} & = &
 \frac{d}{d\sigma} \left(\frac{t}{4}\log (\frac{\sigma^2 + t^2}{4}) + 
 \frac{\sigma-1}{2}\arctan(\frac{t}{\sigma}) + O(\frac{1}{|\sigma+it|})\right) \nonumber \\
& = & 
\frac{t}{4} \left(\frac{2\sigma}{\sigma^2+t^2} \right) + \foh \arctan(\frac{t}{\sigma}) + 
O(\frac{1}{|\sigma+it|}). \label{ln:terms}
\eea
As $\sigma^2+ t^2 \ge 2 \sigma t$, we can bound the first term of (\ref{ln:terms}) by $1/4$, and the second term by 
$\pi/4$, giving
$$
|\frac{d \arg \Gamma(\frac{\sigma + it}{2})}{d\sigma}| \le \frac{1+\pi}{4}+
 O\left(\frac{1}{|\sigma+it|}\right).
$$

Thus, for any $\delta > 0$, if $C = C(\delta)$ is chosen large enough, then for all $\sigma_0 > C$, 
$$
\int_{\sigma_0}^{\sigma_0+2} |\frac{d \arg \xi(\sigma + it)}{d\sigma}| d \sigma \leq \frac{1+\pi}{2} + \delta.
$$

\end{proof}

In a region where  $\sigma/t \to 0$, the parameter range relevant to this paper,
the  first term on the right in
(\ref{ln:terms}) goes to zero, and the upper bound in Lemma \ref{le43}\,(3) above
can be further improved to $\frac{\pi}{2} + \delta$. 
This latter bound cannot be improved, since when $s=\sigma+it$ has $\sigma$ much smaller than $t$
 the argument must necessarily change by nearly $\pi/2$;
this variation comes from the
change in argument of the factor $\Gamma(\frac{s}{2}) $ by $\arg(\frac{s_0}{2})$
 between $s= s_0$ and $s_0 + 2$. \medskip

For later use we  collect some properties of the function $F(\sigma, t)$ 
defined in (\ref{eq333}) above.

\begin{lem}\label{le44}
On the region $\sigma \ge 0$  the
function $F(\sigma, t) $ has the following properties.

 (1) For fixed $t \ge  2\pi e$, the  function $F(\sigma, t)$ is 
 a strictly increasing function of $\sigma$.
 
 (2) For fixed $t \ge 4\pi e$, and any positive $\sigma_1$,
 \begin{equation}\label{eq444}
 \int_{0}^{\sigma_1} F(\sigma, t) \, d \sigma \le 4 F(\sigma_1, t).
 \end{equation}
\end{lem}

\begin{proof}
(1) Rewrite
\begin{equation} \label{eq445}
F(\sigma, t)= \sqrt{\pi}(\sigma^2+t^2)^{\frac{3}{4}} \exp \left( \frac{\sigma}{4} \left(\log ( \sigma^2+ t^2) - \log (4\pi^2 e^2) \right)\right)
\exp\left( -\frac{\pi t}{4} + \frac{t}{2} \arctan(\frac{\sigma}{t})\right).
\end{equation}
For fixed 
$t \ge  2\pi e$ all terms separately in this product are constant or increasing functions
of $\sigma$.

(2) Since 
$t \ge  2\pi e$, by (1) the integrand on the left side of  (\ref{eq444}) is increasing. 
We obtain for $t \ge 4 \pi e$,
\begin{eqnarray*}
\int_{0}^{\sigma_1} F(\sigma, t) d \sigma 
& \le &  \sqrt{\pi} (\sigma_1^2 + t^2)^{\frac{3}{4}}
e^{-\frac{\pi t}{4}+ \frac{t}{2} \arctan (\frac{\sigma_1}{t})}
\int_{0}^{\sigma_1} \exp \left(\frac{\sigma}{4} \log (\frac{\sigma_1^2 + t^2}{4\pi^2 e^2})  \right) d\sigma\\
&\le& \frac{4}{\log (\frac{\sigma_1^2 + t^2}{ 4\pi^2e^2})}  \,  F(\sigma_1, t)\le 4 \, F(\sigma_1, t).
\end{eqnarray*}
\end{proof}

\section{Integrals of the  $\xi$-Function: Proof of Theorem~\ref{th21}}

We consider 
 the family of functions $\Xi_{\lambda}^{(-1)}(z)$ given by (\ref{eq122}),
which has $\Xi_{0}^{(-1)}(z)=-i \xi^{(-1)}(\frac{1}{2} + iz)$.

(1) To  show $\lim_{t \to \infty} \Xi_{\lambda}^{(-1)}(t) = \pi \Phi(0)$ we 
will establish the stronger  result that for $t \ge 3$ one has 
\begin{equation}\label{eq451a}
\Xi_{\lambda}^{(-1)}(t) = \pi \Phi(0) + O \left(\frac{1}{t^{2/3}}\right),
\end{equation}
where the implied constant in the $O$-symbol depends on $\lambda$.
We start from
\begin{eqnarray*}
\Xi_{\lambda}^{(-1)}(t) &= &2\int_{0}^{\infty} e^{\lambda u^2} \Phi(u) \frac{\sin tu}{u} du \\
&=& 2\int_{0}^{\infty} 
e^{\lambda (\frac{v}{t})^2} \Phi(\frac{v}{t})\, \frac{\sin v}{v} dv\\
\end{eqnarray*}
We estimate  the latter integral by splitting
the integration region into  three pieces, the first integrating over 
the interval $[0, 2 \pi \lfloor t^{2/3}\rfloor ]$,
the second integrating over the interval $[2\pi \lfloor t^{2/3} \rfloor, 2 \pi \lfloor t^{4/3}\rfloor ]$, and the
third integrating over  $[2\pi \lfloor t^{4/3} \rfloor, +\infty)$. The first integral will give the main
contribution $\frac{\pi}{2} \Phi(0) + O(\frac{1}{t^{2/3}})$, the second will be bounded
by $O(\frac{1}{t^{2/3}})$, and the third will be shown negligibly small, of size $O(e^{-2t})$.

To obtain the estimates, we view 
$\lambda$ as  fixed and let   $F(u) = e^{\lambda u^2} \Phi(u)$.
In the following estimates, all  $O$-symbols will depend on $\lambda$ unless otherwise noted.
Now $|F(u)|,\, |F^{'}(u)|, \, |F^{''}(u)|$ are all absolutely bounded
on $[0, \infty)$, using the very rapid decrease of $\Phi(u)$ and its first two 
derivatives; this follows from results in Lemma \ref{le41}.
Next,  Lemma ~\ref{le41} (2) shows $F(u)$ is an even function, whence
its power series expansion at $u=0$ gives, for $0 \le u \le 2\pi$,
\begin{equation}\label{eq452a}
F(u) = \Phi(0) + O\left( \Phi^{''}(0) u^2\right).
\end{equation}
For any $u_0 \ge 0$ and $0 \le x \le 2\pi$  we have
\begin{equation}\label{eq453a}
F(u_0+x) = F(u_0) + F'(u_0)x + O \left( |F''(u_0)| x^2\right),
\end{equation}
with the $O$-constant depending on $\lambda$ but not on $u_0$.

For the first integral, on the range $v \in [0, 2 \pi \lfloor t^{2/3}  \rfloor]$, (\ref{eq452a}) gives
$$
F\left(\frac{v}{t}\right) = \Phi(0) + O\left((\frac{v}{t})^2\right).
$$
We obtain
\begin{eqnarray*}
\int_{0}^{2 \pi \lfloor t ^{2/3}\rfloor} F\left(\frac{v}{t}\right) \frac{ \sin v}{v} dv 
&= &\Phi(0) \int_{0}^{2 \pi \lfloor t^{2/3} \rfloor} \frac{\sin v}{v} dv
+ O \left( \int_{0}^{2 \pi \lfloor t^{2/3} \rfloor} \frac{v}{t^2}|\sin v| dv\right) \\
& = & \Phi(0) \int_{0}^{2 \pi \lfloor t^{2/3} \rfloor} \frac{\sin v}{v} dv
+ O \left( \frac{1}{t^{2/3}} \right).
\end{eqnarray*}
Next we use the evaluation of the improper integral
$$
\int_{0}^{\infty} \frac{ \sin u}{u} du := \lim_{T \to \infty} \int_{0}^T \frac{\sin u}{u} du = \frac{\pi}{2}.
$$
We use the quantitative estimate that  for  real $T \ge 1$
$$
\int_{0}^{2 \pi T} \frac{\sin u}{u} du= \frac{\pi}{2} + O\left(\frac{1}{T}\right),
$$
which can be proved by integration by parts. Substituting this in the last equation, 
taking $T = 2\pi \lfloor t^{2/3} \rfloor$,
we obtain
$$
\int_{0}^{2 \pi \lfloor t ^{2/3}\rfloor} F\left(\frac{v}{t}\right) \frac{ \sin v}{v} dv =
 \frac{\pi}{2} \Phi(0) + O \left( \frac{1}{t^{2/3}}\right).
$$

For the second integral, we have
\begin{eqnarray*}
\int_{2\pi \lfloor t^{2/3}\rfloor}^{2 \pi \lfloor t^{4/3} \rfloor} F\left(\frac{v}{t}\right) \frac{\sin v}{v} dv
&= & \sum_{n = \lfloor t^{2/3} \rfloor}^{\lfloor t^{4/3} \rfloor-1}
 \int_{2\pi n}^{2\pi(n+1)} 
F\left( \frac{v}{t}\right) \frac{\sin v}{v} dv \\
&=& \sum_{n = \lfloor t^{2/3} \rfloor}^{\lfloor t^{4/3} \rfloor-1}
2 \pi  \int_{0}^{1} F\left( \frac{2\pi (n+x)}{t}\right) \frac{\sin 2\pi x}{2 \pi(n+x)} dx. 
\end{eqnarray*}
For $n \ge 2$ and $0 \le x \le 1$ we have 
$$
\frac{1}{2\pi (n+x)}= \frac{1}{2\pi n}\left( 1 + O\left(\frac{x}{n}\right) \right),
$$
where the $O$-constant is absolute. This yields
$$
\int_{0}^1 F\left(\frac{2\pi(n+x)}{t}\right) \frac{\sin 2 \pi x}{2\pi(n+x)} dx
= \frac{1}{2 \pi n} \int_{0}^1 F\left( \frac{2\pi(n+x)}{t}\right) \sin 2 \pi x \, dx + O\left( \frac{1}{n^2}\right),
$$
where the $O$-constant depends on $\lambda$ but not on $n \ge 2$. 
We now put in the integral on the right the bound, obtained from (\ref{eq453a}), that
for $0 \le x \le 1$, 
$$
F\left(\frac{ 2\pi (n+x)}{t}\right) = F\left(\frac{2 \pi n}{t}\right) + O\left( |\frac{x}{t}| + |\frac{x}{t}|^2\right).
$$
Substituting this in the integral, the constant term $F(\frac{2\pi n}{t})$ integrates to $0$,
and we obtain
$$
\frac{1}{2 \pi n} \int_{0}^1 F\left( \frac{2\pi(n+x)}{t}\right) \sin 2 \pi x dx = O\left( \frac{1}{nt} \right).
$$
We conclude that
$$
\int_{2\pi \lfloor t^{2/3}\rfloor}^{2 \pi \lfloor t^{4/3} \rfloor} F\left(\frac{v}{t}\right) \frac{\sin v}{v} dv
= O \left(  \sum_{n = \lfloor t^{2/3} \rfloor}^{\lfloor t^{4/3} \rfloor} ( \frac{1}{nt} + \frac{1}{n^2}) \right) 
= 
O\left( \frac{1}{t^{2/3}}\right).
$$

For the third integral, we use the rapid decrease of $\Phi(u)= O( e^{- e^u})$ to conclude,
with much to spare, that
$$
\left|\int_{ \lfloor t^{4/3} \rfloor} F\left(\frac{v}{t}\right) \frac{\sin v}{v} dv\right| \le \int_{t^{1/3}}^{\infty} F(u) du \le e^{-2t}.
$$

Combining these three integral estimates, we obtain the desired bound (\ref{eq451a}). \medskip

(2) First, the  fact that  $\lim_{t \to \infty} \Xi_{\lambda}^{(-1)}(t) = A_0  \ne 0$
implies that the function $\Xi_{\lambda}^{(-1)}$ has at most finitely many zeros on the positive $t$
axis. The functional equation $\Xi_{\lambda}^{(-1)}(-t) = - \Xi_{\lambda}^{(-1)}(t)$ gives the result
on the negative $t$ axis as well, and shows $\Xi_{\lambda}^{(-1)}(0)=0.$
Since these functions are entire of order $1$ and maximal type by
Lemma~\ref{le42}(1), they necessarily have infinitely many zeros, whence all
but finitely many are complex zeros.

Secondly, we recall that in  1947 Wintner \cite{Wi47} proved directly that 
\begin{equation}\label{eq454}
\Xi^{(-1)}(t) :=\Xi_{0}^{(-1)}(t) >0 ~~ \mbox{when} ~~t>0;
\end{equation}
this fact implies that $\Xi^{(-1)}(z) = \Xi^{(-1)}_0(z)$ has
no zeros on the positive real axis, and the functional equation gives the same
on the negative real axis.
Here we note in passing that $\Xi^{(-1)}(z)$
has a simple zero at $z=0$, since $\Xi(0) = \xi(\frac{1}{2}) \approx 0.49712 \ne 0$. 
Wintner's proof is based on  the following assertion.

{\bf Claim.} {\em  If a function $\Psi(u)$ is positive and decreasing
on the positive real axis, then for each positive $t$, }
\begin{equation*}
\int_{0}^{\infty} \Psi(u) \left(\frac{ \sin tu}{u} \right) du = \lim_{X \to \infty} \int_{0}^X  \Psi(u)
\left( \frac{ \sin tu}{u} \right)
du >0.
\end{equation*}

To prove the claim, the existence of the limit is seen
by writing $\psi(u) =\frac{1}{u} \Psi(u)$ and noting it 
is positive and decreases to $0$ at $\infty$. On choosing
values  $X=X_n= \frac{n\pi}{t}$ one has
$$
  \int_{0}^{X_n}  \Psi(u) \frac{ \sin tu}{u} du =  
   \frac{1}{t} \sum_{k=1}^{n} \int_{(k-1)\pi}^{k \pi} \psi(\frac{v}{t})\sin v \, dv.
$$
Observing  that the terms of the series have alternating signs, are decreasing and go to zero,
one can let $n \to \infty$, get a convergent series, which has a positive limit since its
first term is positive.  The limit  exists over  all $X$ since the variation between $X_n\le X \le X_{n+1}$
goes to zero as well.
This proves the claim.

We now choose $\Psi(u)= \Phi(u)$ in the claim,  noting that  Lemma ~\ref{le41}(5) asserts 
the positive decreasing hypothesis holds, and the claim gives Wintner's result
(\ref{eq454}). The functional equation in Lemma ~\ref{le41}(2) then gives
$\Xi_{0}^{(-1)}(t)<0$ when $t < 0$, which proves assertion (2) in the 
case $\lambda=0$.
It is immediate that $e^{\lambda u^2} \Phi(u)$ is also positive and decreasing
for  $\lambda \le 0$, whence the claim  applies similarly to establish (2).
$~~~\Box$\\


\section{Value Sets of $\xi^{(-1)}(s)$: Proof of Theorem~\ref{th22}}

The basic idea behind the lower bound in  this theorem is given by the following two facts.
\begin{enumerate}
\item
As $t \to \infty$  the function
$\xi^{(-1)}(\sigma+it)$ approaches $i A_0$  uniformly on any  vertical
strip $\sigma_1 \le \sigma \le \sigma_2$, where $-\infty < \sigma_1 < \sigma_2 < \infty$.
Thus  for any value $\cst \ne iA_0$ all the solutions to $\xi^{(-1)}(s) = \cst$
in the strip must lie below some finite bound $t \le C$, where $C$ depends on
$\sigma_1, \sigma_2$. 

\item As $t \to -\infty$
the  function $\xi^{(-1)} (\sigma+it)$ approaches $-i A_0$ uniformly on any vertical
strip $\sigma_1 \le \sigma \le \sigma_2$. Thus  for any value $\cst \ne -iA_0$ all the solutions to \\
$\xi^{(-1)}(s) = \cst$
in the strip must lie above some finite bound $t \ge -C$, where $-C$ depends on
the values $\sigma_1, \sigma_2$. 
\end{enumerate}
These two facts follow directly from Theorem~\ref{th21}, using the well known fact 
that $|\xi(s)| \to 0$ as $|t| \to \infty$, uniformly on any vertical strip.
(This may be proved following Lemma \ref{le43}(1).)
We fix a vertical strip,
which without loss of generality includes the line $Re(s) = \frac{1}{2}$ in
its interior.
Then for  $\sigma_0 \in [\sigma_1, \sigma_2]$ we have
$$
\xi^{(-1)}(\sigma_0+it) = \xi^{(-1)}(\frac{1}{2} + it) + \int_{1/2}^{\sigma_0} \xi(\sigma+ it) d\sigma.
$$
Theorem ~\ref{th21} now gives  $\xi^{(-1)}(\frac{1}{2} + it) \to iA_0$,
 as $t \to \infty$, and $ \xi^{(-1)}(\frac{1}{2} + it) \to -iA_0$ as $t \to -\infty.$
 The uniform bound on $|\xi(s)| \to 0$ as $|t| \to \infty$  in the strip then 
shows that the  integral on the right can be bounded by $\epsilon$ in absolute value for
large enough $|t|$ (depending on $\sigma_1, \sigma_2$), and the facts follow.

The proof of Theorem \ref{th22}  obtains  a lower bound 
using  a quantitative version of the two facts above,  
determining the dependence of the constants $C$ above on  the width of
the strip, chosen to have $Re(s)=\frac{1}{2}$ as its central line. 
The upper bound is  obtained by analyzing the rapid growth of
$\xi(s)$ on horizontal lines of constant $t$, which comes from  the gamma
 factor in $\xi(s)$. \medskip

We commence the proof.
Using the symmetries of the function $\xi^{(-1)}(s)$, it suffices to prove the results (1) and (2)
for a zero $\xi^{(-1)}(\rho)= \cst$ with $\rho= \sigma + it $
in  the first quadrant region $\sigma \ge \frac{1}{2}$ and $t \ge 0$.
In this proof we treat $\cst$ as  fixed, and all constants $C_j$ given in the proof
will depend on $\cst$. We divide the first quadrant region into three subregions
which we treat separately.

The first case considers  the subregion    $\frac{1}{2} \le \sigma \le 2$, and $t \ge 0$.
 We assert that
$|\xi^{(-1)}(s) - iA_0| \to 0$  as $t \to \infty$ uniformly in this range of $\sigma$.
From the assertion we may conclude
that for any $\cst \ne iA_0$ the solutions to $\xi^{(-1)}= \cst$ are confined
to a compact region, which proves the theorem in this case.
The assertion immediately follows  from the result of Theorem ~\ref{th21}(1) 
(for $\lambda=0$)  that gives $\xi^{(-1)}(\frac{1}{2} + it) \to i A_0$
as $t \to \infty$, combined with the bound
$$
|\xi^{(-1)}(\sigma + it) - \xi^{(-1)}(\frac{1}{2}+ it)| \le C_{4} e^{-\frac{\pi}{4} |t|} (|t|+1)^{\frac{5}{2}},
$$
which follows from  Lemma~\ref{le43}(1) by integration on a horizontal line.

The second case is the subregion  $\sigma \ge 2$ and $t \ge  4 \pi e$, which is the main case.
To prove the bounds (1) and (2) for  this case, we will obtain lower and upper bounds
of the required form  on $\sigma$ as a function of $t$.
The lower bound (for $\cst \ne iA_0$)
asserts  there is a constant $C_{5}$ (depending on $\cst$) 
such that $|\xi^{(-1)} (s_0)-\cst| \ne 0$
for $s_0= \sigma_0+ it$,  whenever 
\begin{equation}\label{eq621}
\frac{1}{2} \le \sigma_0 \le \frac{\pi}{2}\left( \frac{t}{\log t} \right) - C_{5} \frac{t}{(\log t)^2}.
\end{equation}
We begin with
$$
\xi^{(-1)}(s_0) = \xi^{(-1)}(\frac{1}{2} +it) + \int_{\frac{1}{2}}^{\sigma_0} \xi (\sigma+ it ) d \sigma.
$$
We assume $\cst \ne iA_0$, and set $\delta= |\cst- i A_0|>0.$
We have $\xi^{(-1)}(\foh+iT) \to iA_0$ as $T \to \infty$, and thus 
one  has $|\xi^{(-1)} (\foh+iT)- \cst| \ge  \frac{1}{2} \delta$ for all $T$ larger than some constant $T_0$. Note, however, that we can assume that $|\xi^{(-1)} (\foh+it)- \cst| \ge  \frac{1}{2} \delta$ holds for all pairs $\sigma + it$ satisfying (\ref{eq621}) with $t \ge 4\pi e$ 
since we can increase the size of $C_{5}$ such that (\ref{eq621}) 
will have no solutions for $t < T_0$.

We next show one can pick $C_{5}$ large enough  that when $\sigma_0$ satisfies  (\ref{eq621})
we have the  estimate, valid for $t  \ge 4 \pi e$,
\begin{equation}\label{eq623}
 \int_{\frac{1}{2}}^{\sigma_0} |\xi (\sigma+ it )| \,d \sigma \le \frac{1}{4} \delta.
\end{equation}
We use Lemma \ref{le43} (1) to bound  the integral from $\sigma=\frac{1}{2}$
to $\sigma=2$ by $\frac{1}{8}\delta$.
For the remaining integral with
$\sigma$ satisfying (\ref{eq621}) we use the upper bound in Lemma \ref{le43}(2),
noting that in this range
 $$
 \arctan \left(\frac{t}{\sigma} \right)= \frac{\pi}{2} - \frac{\sigma}{t} +O \left(\frac{\sigma^2}{t^2} \right),
 $$
to obtain 
\begin{eqnarray*}
 \int_{2}^{\sigma_0} |\xi (\sigma+ it )| d \sigma
 &\le & C_{6} |t|^{3/2} 
 \int_{2}^{\sigma_0} \exp\left(\frac{\sigma}{2}\log |\frac{\sigma+it}{2\pi e}| -\frac{\pi}{4}t+ O(t/\log t) \right) d\sigma\\
& \le & 
C_{7} |t|^{3/2} e^{-\frac{\pi}{4} t}
 \int_{2}^{\sigma_0} \exp\left(\frac{\pi}{4} (\frac{t}{\log t}) \log t - C_{5} \frac{t}{\log t}+ O(\frac{t}{\log t}) \right) d\sigma\\
& \le & C_{8} |t|^{5/2}
  \exp\left( - C_{5} \frac{t}{\log t}+ O(\frac{t}{\log t}) \right). \\
  \end{eqnarray*}
Keeping in mind that we are free to choose $C_{5}$ as large as necessary, we note that it is possible to choose $C_{5}$ large enough, depending on $\cst$,
to make the exponential term above smaller than $\exp(-C_{9} \frac{t}{\log t})$,
where $C_{9}$ is large enough that $\exp(-C_{9} \frac{t}{\log t}) \le
\frac{1}{8}\delta$ for all $t \ge 4 \pi e$. Thus we establish (\ref{eq623}).

Now the triangle inequality gives

\bea
|\xi^{(-1)}(s_0) - \cst| & \ge & |\cst - iA_0| - |iA_0 - \xi^{(-1)}(\foh + it)| - |\xi^{(-1)}(\foh + it) - \xi^{(-1)}(s_0)| \nonumber \\
& \ge & \delta - \frac{\delta}{2} - \frac{\delta}{4} > 0,
\eea
as asserted. This bound applies to all $\cst \ne  iA_0$, for $t \ge 4 \pi e$ and it similarly applies
for $\cst \ne - iA_0$   in the lower half-plane region $t \le -4 \pi e$.
Thus it gives the lower bound asserted in (1), for $\cst \ne \pm iA_0$.
The upper bounds
in (1) and (2) assert that there is a constant $C_{10}$ (depending on $\cst$)  such that 
for any fixed $\cst$ (including $\cst= \pm iA_0$) one
has $|\xi^{(-1)} (s_0)-\cst| \ne 0$ whenever 
\begin{equation}\label{eq632}
 \sigma_0 \ge \frac{\pi}{2} \left( \frac{t}{\log t} \right) +C_{10} \frac{t}{(\log t)^2}.
\end{equation}
It suffices to prove that $|\xi^{(-1)} (s_0)|> |\cst|$ holds when
(\ref{eq632}) holds. 
To show this upper bound, we will use the following analytic lemma.

\begin{lem} \label{lem:arg}
Suppose that $f:[0,\infty) \rightarrow \mathbb{C}\setminus \{0\}$ is continuous, and that the total variation of $\arg(f)$ on the interval $[a,b]$ is at most  $\theta < \pi$. Then
\be \left|\int_a^b f(x)dx\right| \geq \cos\left(\frac{\theta}{2}\right) \int_a^b |f(y)|\d y \ee
\end{lem}

\begin{proof}
Note that since the total variation of $\arg(f)$ on the interval $[a,b]$ is equal to $\theta < \pi$, there exists $\beta \in \mathbb{R}$ such that $\arg(f(x)) \in [\beta-\frac{\theta}{2}, \beta + \frac{\theta}{2}]$ for $x \in [a,b]$. Then for $v_1 = e^{i\beta}$ and $v_2 = e^{i(\beta+\pi/2)}$, there exist real valued functions $u$ and $w$ such that $f(x) = u(x)v_1 + w(x)v_2$. Since $v_1$ and $v_2$ are orthogonal, 
$$
\left|\int_a^b f(y) d y \right| = \left|\left(\int_a^b u(x) d x\right)v_1 + \left(\int_a^b w(y) d y\right)v_2\right| \geq \left|\int_a^b u(x) d x\right|.
$$
Finally, since $\arg(f) \in [\beta - \frac{\theta}{2}, \beta + \frac{\theta}{2}]$, we have that $u(x) \geq \cos(\frac{\theta}{2}) |f(x)|$, so 
$$
\left|\int_a^b f(y) d y \right| \geq \left|\int_a^b u(x) d x\right| \geq \cos\left(\frac{\theta}{2}\right) \int_a^b |f(x)| d x.
$$
\end{proof}

We write $\xi^{(-1)}(\sigma_0 + it) = \xi^{(-1)}(\foh+it) + \int_{1/2}^{\sigma_0} \xi(\sigma+it) d\sigma,$
and will use the fact that the main contribution to the size of $\xi^{(-1)}(\sigma+it)$ will come
from the integral over a small interval near its right endpoint, and the function will be
very large when (\ref{eq632}) holds.
Thus we start from the inequality
\begin{equation*}
 |\xi^{(-1)}(\sigma_0 +it)|  \geq \left| \int_{\sigma_0-2}^{\sigma_0} \xi(\sigma+it)\d \sigma \right|  - \int_{1/2}^{\sigma_0-2} |\xi(\sigma+it)| \d \sigma - |\xi^{(-1)}(\foh+it)|,
\end{equation*}
and will show that the right side is positive when (\ref{eq632}) holds.
A total variation  bound on $\arg(\xi(\sigma+it))$ is obtained via
Lemma \ref{le43}(3), taking $\delta = \frac{2\pi}{3} - (\frac{\pi+1}{2}) \approx 0.0236$, 
yielding
$$
\int_{\sigma_0-2}^{\sigma_0} \left| \frac{d \arg \xi(\sigma+it)}{d\sigma}\right|\d \sigma \leq \frac{2\pi}{3}.
$$ 
which is valid provided $\sigma_0 \ge C_{11}$ for a suitable constant $C_{11}$.
Now  Lemma \ref{lem:arg} applies to give
$$
\left| \int_{\sigma_0-2}^{\sigma_0} \xi(\sigma+it)\d \sigma \right| \ge 
\frac{1}{2} \int_{\sigma_0-2}^{\sigma_0}\left| \xi(\sigma+it)\right| \d \sigma. 
$$
We conclude that
\begin{equation} \label{eq670}
 |\xi^{(-1)}(\sigma_0 +it)|  \geq \frac{1}{2} \int_{\sigma_0-2}^{\sigma_0} \left|\xi(\sigma+it)\right| \,\d \sigma   - \int_{1/2}^{\sigma_0-2} |\xi(\sigma+it)| \d \sigma -  |\xi^{(-1)}(\frac{1}{2}+ it)|.
\end{equation}
The last term on the right has $|\xi^{(-1)}(\frac{1}{2}+it)| = O(1)$, since $\xi^{(-1)}(s)$
is bounded on the critical line.

We now obtain  from Lemma \ref{le43} (2) and Lemma \ref{le44} (1) a lower
bound for the first integral on the right hand side of (\ref{eq670}). 
Namely, for all sufficiently large $\sigma_0$ and $t \ge 4 \pi e$ there holds
\begin{eqnarray*}
\int_{\sigma_0-2}^{\sigma_0} \left|\xi(\sigma+it)\right| \,\d \sigma
& \ge & \frac{1}{2}  \int_{\sigma_0-1}^{\sigma_0} F(\sigma, t) \,\d \sigma\\
& \ge & \frac{1}{2} F(\sigma_0 -1, t),
\end{eqnarray*}
where $F(\sigma, t)$ is given by (\ref{eq333}).
On the other hand, using Lemma \ref{le43}(2)
and  Lemma \ref{le44}(2), we can pick a  constant $C_{12}$ 
large enough 
that for all $t \ge 4 \pi e$ and sufficiently large $\sigma_0$ (depending on $C_{12}$),
\begin{eqnarray*}
\int_{1/2}^{\sigma_0-2} \left|\xi(\sigma+it)\right| \,\d \sigma
& \le & C_{12} \int_{1/2}^{\sigma_0-2} F(\sigma, t) \,\d \sigma\\
& \le  & C_{13} F(\sigma_0-2, t).
\end{eqnarray*}

The function $F(\sigma, t)$ is  rapidly increasing in $\sigma$.
For $t \ge 4 \pi e$ and $\sigma_0 > 2$  comparison of the terms in (\ref{eq445}) yields
\begin{equation}\label{eq561}
\frac{F(\sigma_0-1, t)}{F(\sigma_0 -2, t)}
 \ge \frac{1}{\sqrt{2 \pi e}} ((\sigma_0-2)^2 + t^2)^{\frac{1}{4}}.
\end{equation}
This fact implies
for all sufficiently large $\sigma_0$,
 $\frac{1}{4}F(\sigma_0 -1, t) \ge 4C_{13} F(\sigma_0 -2, t),$ whence half of the first term on the
 right in (\ref{eq670}) already dominates the second integral. It remains to choose
 $\sigma_0$ large enough as a growing function of $t$ that the remaining half of the 
 absolute value of the first term also dominates
 $|\xi^{(-1)}(\frac{1}{2}+ it)| +|\cst|= O(1)$ on the right side of (\ref{eq670}).
  We show the  lower bound
 in  (\ref{eq632}) achieves this, taking \\
 $\sigma_0 \ge \frac{\pi}{2}(\frac{t}{\log t}) + C_{5} \frac{t}{(\log t)^2}$
 with sufficiently large $C_{5}$. By the monotonicity property in Lemma \ref{le44}(1) it suffices to consider 
 $\sigma_0 = \frac{\pi}{2}(\frac{t}{\log t})+ C_{5} \frac{t}{(\log t)^2}$, for which we obtain 
 \begin{eqnarray*}
  F(\sigma_0 - 1, t) &\ge & 
  \exp \left( - \frac{\pi \, \log (2\pi e)}{4}(\frac{t}{\log t})+O( \frac{t}{(\log t)^2} )
  \right) \cdot
   \exp \left( \frac{\pi t}{4} + C_{5} \frac{t}{(\log t)}\right) \cdot \\
 && \cdot \exp \left( -\frac{\pi t}{4}  + \frac{\pi}{4}(\frac{t}{\log t}) + 
 O( \frac{t}{(\log t)^2} ) \right)\\
 & \ge &  \exp \left( C_{10} \frac{t}{\log t} + O(\frac{t}{(\log t)^2} \right).
  \end{eqnarray*}
Here $C_{10}= C_{5} - \frac{\pi}{4} \log (2\pi)$, so  by choosing $C_{5}$ sufficiently large, 
we can overcome the $O$-constant
in the last term and 
force  
$$
\exp \left( C_{10} \frac{t}{\log t} + O(\frac{t}{(\log t)^2}\right) > |\xi^{(-1)}(\foh+it)|+|\cst|
$$
to hold  for all $t \geq 4 \pi e$. 
We conclude that for proper choices of $C_{12}, C_{5}$ the right hand side
of (\ref{eq670}) is larger than$|\cst|$ for $t \geq 4 \pi e$ and $\sigma$ satisfying (\ref{eq632}). 
This gives, for any fixed $\cst \in \C$,
an estimate establishing the upper bound case of both (1) and (2) in the range
$t \ge 4 \pi e$.

The third case  is the subregion   $0 \le t \le 4\pi e$ and $\sigma \ge 2$. We assert  
 that $| \xi^{(-1)}(s)|$ becomes very large as $\sigma$ increases,
which for any constant $\cst \in \C$  will confine
solutions to $\xi^{(-1)}(s)= \cst$  with $0 \le t \le 4 \pi e$
to a compact region $0 \le \sigma \le C_{14}$,
and so complete the proof. 

We proceed to estimate the size of $\xi^{(-1)}(s_0)$,
with $s_0 = \sigma_0+it$,  using
$$
\xi^{(-1)}(s_0) = \xi^{(-1)} (\sigma_0 -2) + \int_{0}^t \xi(\sigma_0 - 2 + i y) dy +
\int_{\sigma_0 - 2}^{\sigma_0} \xi(\sigma + it) d\sigma.
$$
We obtain
\bea
|\xi^{(-1)}(s_0)|  &\ge& |\int_{\sigma_0 - 2}^{\sigma_0} \xi(\sigma + it) d\sigma|
- |\xi^{(-1)}(\sigma_0-2)| - \int_{0}^t |\xi(\sigma_0 - 2 + i y)| dy \nonumber \\
&\ge&  |\int_{\sigma_0 - 2}^{\sigma_0} \xi(\sigma + it) d\sigma|
-  \Big(|\xi^{(-1)}(\sigma_0-2)| + 4\pi e |\xi(\sigma_0-2)|\Big),\label{ln:xineq}
\eea
with the last inequality based on the fact that 
for fixed $\sigma$, the function $|\xi(\sigma +it)|$ is maximized on the
real axis. The argument now proceeds similarly to the case $t \ge 4 \pi e$.
Namely, the first integral can be bounded below by the use of Lemma \ref{lem:arg}.
Together with Lemma ~\ref{le43}(2) we obtain
\begin{eqnarray*}
|\int_{\sigma_0 - 2}^{\sigma_0} \xi(\sigma + it) d\sigma| &\ge&
\frac{1}{2} \int_{\sigma_0 -2}^{\sigma_0} |\xi(\sigma + it)| d \sigma \\
& \ge & \frac{1}{2}  \int_{\sigma_0-1}^{\sigma_0} F(\sigma, t) \,\d \sigma\\
& \ge & \frac{1}{2} F(\sigma_0 -1, t),
\end{eqnarray*}
where $F(\sigma, t)$ is given by (\ref{eq333}).
Since $t \le 4 \pi e$ and $\sigma_0 \ge 2$ 
 we have  
$$
F(\sigma_0 -1, t) \ge e^{- \pi^2 e} F(\sigma_0 -1, 0)
$$
directly from the definition (\ref{eq333}).
Additionally, Lemma  ~\ref{le43}(2) 
guarantees the existence of a constant $C_{15}$ such that for all $\sigma_0  \ge C_{15}$,
$$
4\pi e |\xi(\sigma_0-2)| \le (4\pi e+1) F(\sigma_0 -2,0).
$$
We also obtain
$$|\xi^{(-1)}(\sigma_0-2)| \le \int_{1/2}^{\sigma_0 -2} |\xi(\sigma)| d \sigma 
\le 4F(\sigma_0 -2,0) + \int_{1/2}^{4\pi e} |\xi(\sigma)| d \sigma,
$$
by appealing to the estimate
$$
\int_{4 \pi e}^{\sigma_1} F(\sigma, t) d\sigma \le 4 F(\sigma_1, t)
$$
which is seen to be valid for any $t \ge 0$, following the proof of Lemma~\ref{le44}(2).

Combining all of the above estimates with  (\ref{ln:xineq}), we have
$$
|\xi^{(-1)}(s_0)| \ge \frac{e^{-\pi^2 e}}{2} F(\sigma_0-1,0)-\left(\int_{1/2}^{4\pi e}|\xi(\sigma)|d \sigma + (4\pi e + 5)F(\sigma_0-2,0)\right).
$$

Because the integral in the right hand side above is a constant and the function $F(\sigma,0)$ is increasing without bound,
we can choose constants $C_{16}, C_{17} > 0$ such that for all $\sigma_0 > C_{16}$, we have
$$
|\xi^{(-1)}(s_0)| \ge \frac{e^{-\pi^2 e}}{2} F(\sigma_0-1,0)- C_{17} F(\sigma_0-2,0).
$$

Finally, we may apply the bound (\ref{eq561}) to conclude that
for all $\sigma_0 \ge C_{18}$ and $0 \le t \le  4 \pi e$ the first term on the right hand side above dominates the second enough to give
$$
|\xi^{(-1)}(s_0)|  \ge F(\sigma_0-2, 0).
$$

Since the function on the right is unbounded as $\sigma_0$
increases, by choosing $C_{14}$ sufficiently large we can guarantee that  $|\xi^{(-1)}(s_0)| > |\cst|$
holds  on the region $\sigma_0 \ge C_{14},$
$0 \le t \le 4 \pi e$.
This completes the proof of  Theorem \ref{th22}. \medskip

 \paragraph{\bf Remark.} In the exceptional case (2), for $\cst= iA_0$ and $t >0$,
 it seems possible that a stronger upper bound than (\ref{eq223}) may be valid.
  We  cannot even rule out the possibility that a $|\sigma| \le O\left(1 \right)$ upper bound
 might be valid; see  the numerical data in \S5, plotted  in Figure \ref{fig34}.
 To improve the upper bound significantly, one  would like an improved error term in (\ref{eq451a})
 that decreases exponentially in $t$.
%
\section{Numerical Results}\label{sec6}

We  report  on numerical results on
the zeros of $\xi^{(-1)}(s) - c$, kindly supplied to us 
 by Henri Cohen. These results were 
 computed using PARI. 

Table \ref{tab31} below gives values  of the first few zeros of the function $\xi^{(-1)}(s).$
They are distributed in a very regular way, consistent with Theorem \ref{th22}.
For comparison purposes we include data on averaged position of pairs of 
consecutive zeros of $\xi(s)$ (i.e. zeta zeros in the critical strip). Each zero $\rho$ of
$\xi^{(-1)}(s)$ (with $Im(\rho)>0$) off the critical line has a companion zero $1-\bar{\rho}$ and we expect
these to correspond to a pair of notrivial zeta zeros.

%
%
%
  \begin{table}\centering
\renewcommand{\arraystretch}{.85}
\begin{tabular}{|r|r|r|r||r|r|r|}
\hline
\multicolumn{1}{|c|}{$k$} &
\multicolumn{1}{c|}{$\mbox{Re}( \rho_k)$} &
\multicolumn{1}{c|}{$\mbox{Im}( \rho_k)$} &
\multicolumn{1}{c||}{$|\rho_k|$} & 
\multicolumn{1}{c|}{$\tilde{\gamma}_k$} &
\multicolumn{1}{c|}{$\gamma_{2k-1}$} & 
\multicolumn{1}{c|}{$\gamma_{2k}$}
\\ \hline
0 &          0.50000 & 0.00000 i & 0.00000 & & & ~
\\ \hline
  1 & 12.26164 & 10.74143 i   & 16.30111 & 17.57838 i & 14.13472 i &  21.02203 i  \\ \hline
  2 & 16.59401 & 18.18824 i & 24.62059 & 27.71787 i & 25.01085 i &  30.42487 i \\ \hline
  3 & 19.91864 & 24.52433 i & 31.59501 & 35.26062 i & 32.93506 i &  37.58617 i \\ \hline
  4 & 22.76123 & 30.28316 i & 37.88330 & 42.12290 i & 40.91871 i &  43.32707 i  \\ \hline
  5 & 25.30557 & 35.66576 i & 43.73121 & 48.88949 i & 48.00515 i &  49.77383 i \\ \hline
  6 & 27.64154 & 40.77783 i & 49.26344 & 54.70828 i & 52.97032 i &  56.44624 i  \\ \hline
  7 & 29.82109 & 45.68184 i & 54.55391 & 60.08941 i & 59.34704 i &  60.83177 i   \\ \hline
  8 & 31.87747 & 50.41877 i & 59.65087 & 66.09617 i & 65.11254 i &  67.07981 i  \\ \hline
  9 & 33.83352 & 55.01727 i & 64.58799 & 70.80678 i & 69.54640 i &  72.06715 i  \\ \hline
10 & 35.70571 & 59.49838 i & 69.38988 & 76.42477 i & 75.70469 i &  77.14484 i  \\ \hline
11 & 37.50640 & 63.87809 i & 74.07524 & 81.12388 i & 79.33737 i &  82.91038 i  \\ \hline
12 & 39.24515 & 68.16894 i & 78.65868 & 86.08038 i & 84.73549 i &  87.42527 i  \\ \hline
13 & 40.92954 & 72.38096 i & 83.15186 & 90.65050 i & 88.80911 i &  92.49189 i \\ \hline
14 & 42.56569 & 76.52235 i & 87.56431 & 95.26098 i & 94.65134 i &  95.87063 i  \\ \hline
15 & 44.15865 & 80.59992 i & 91.90394 & 100.07452 i & 98.83119 i &  101.31785 i  \\ \hline
16 & 45.71262 & 84.61941 i & 96.17738 & 104.58608 i & 103.72553 i &  105.44662 i  \\ \hline
17 & 47.23115 & 88.58569 i & 100.39027 & 109.09907 i& 107.16861 i &  111.02953 i  \\ \hline
18 & 48.71728 & 92.50297 i & 104.54746 & 113.09744 i& 111.87465 i &  114.32022 i  \\ \hline
19 & 50.17363 & 96.37488 i & 108.65317 & 117.50873 i& 116.22668 i &  118.79078 i  \\ \hline
20 & 51.60248 & 100.20464 i & 112.71107 & 122.15847 i & 121.37012 i &  122.94682 i \\ \hline
 \end{tabular}
\caption{Initial zeros $\rho_k$ of  $\xi^{(-1)}(s)$ in the first quadrant. 
Averaged pairs of zeta  zero ordinates  $\tilde{\gamma}_k :=\frac{1}{2}(\gamma_{2k-1}
+ \gamma_{2k})$
are included for comparison.}
\label{tab31}
\end{table}
 
Figure \ref{fig32} pictures a plot  the first 100 zeros of $\xi^{(-1)}(s)$ in each quadrant;
note the four-fold symmetry, and the fact that the zeros appear to fall on a smooth
curve. (A suitable  smooth curve that interpolates the points is given by a certain level set of the function
$F(s)= \xi^{(-1)}(s)- iA_0.$)

%
%
%

\begin{figure}
\centering
\includegraphics[scale=1.00]{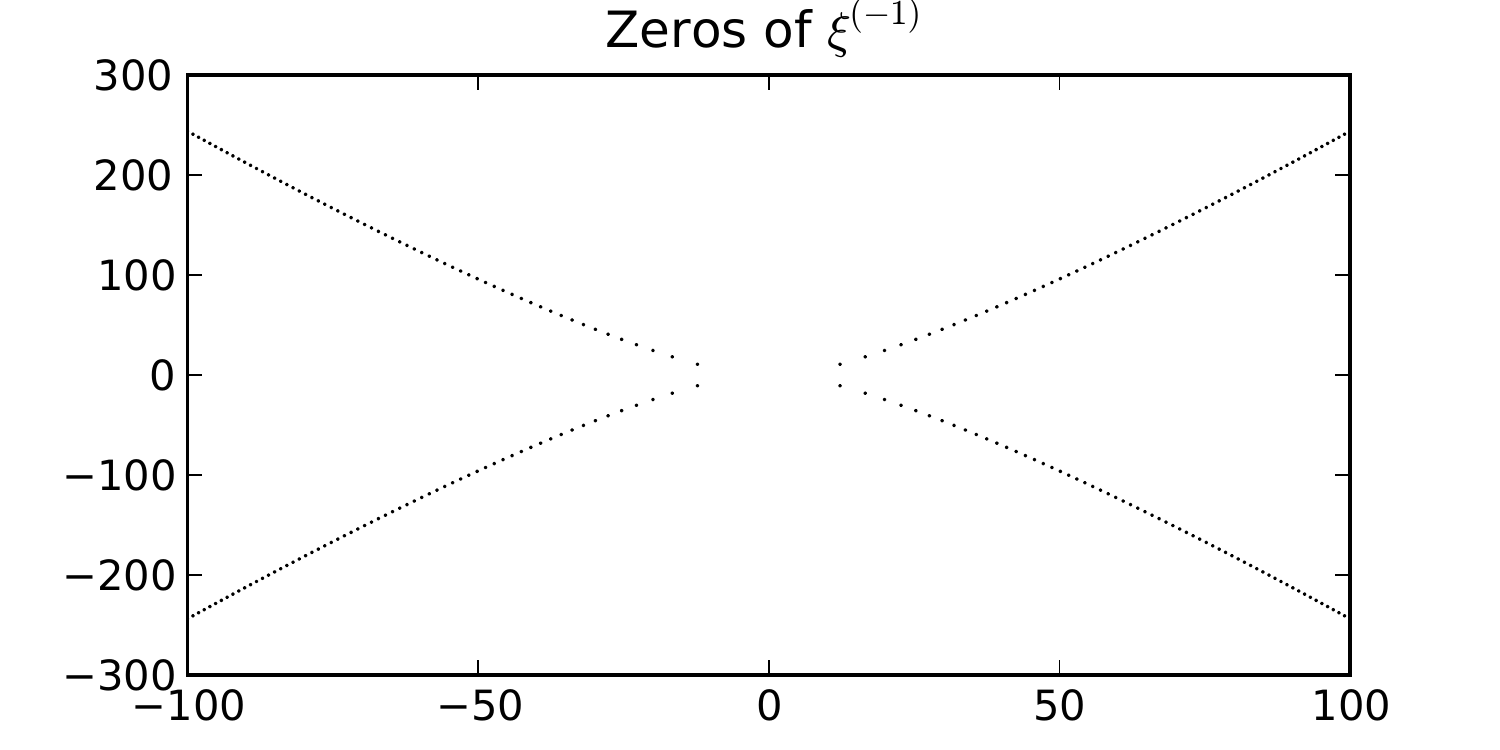}
\caption{
Plot of initial zeros of $\xi^{(-1)}(s)$ in all four quadrants.
}
\label{fig32}
\end{figure}

 Figure \ref{fig33} plots, on a smaller scale, the first $500$ zeros in the
first quadrant. There is
general  agreement with the asymptotics of Theorem \ref{th22}.

%
%
%

\begin{figure}
\centering
\includegraphics[scale=1.00]{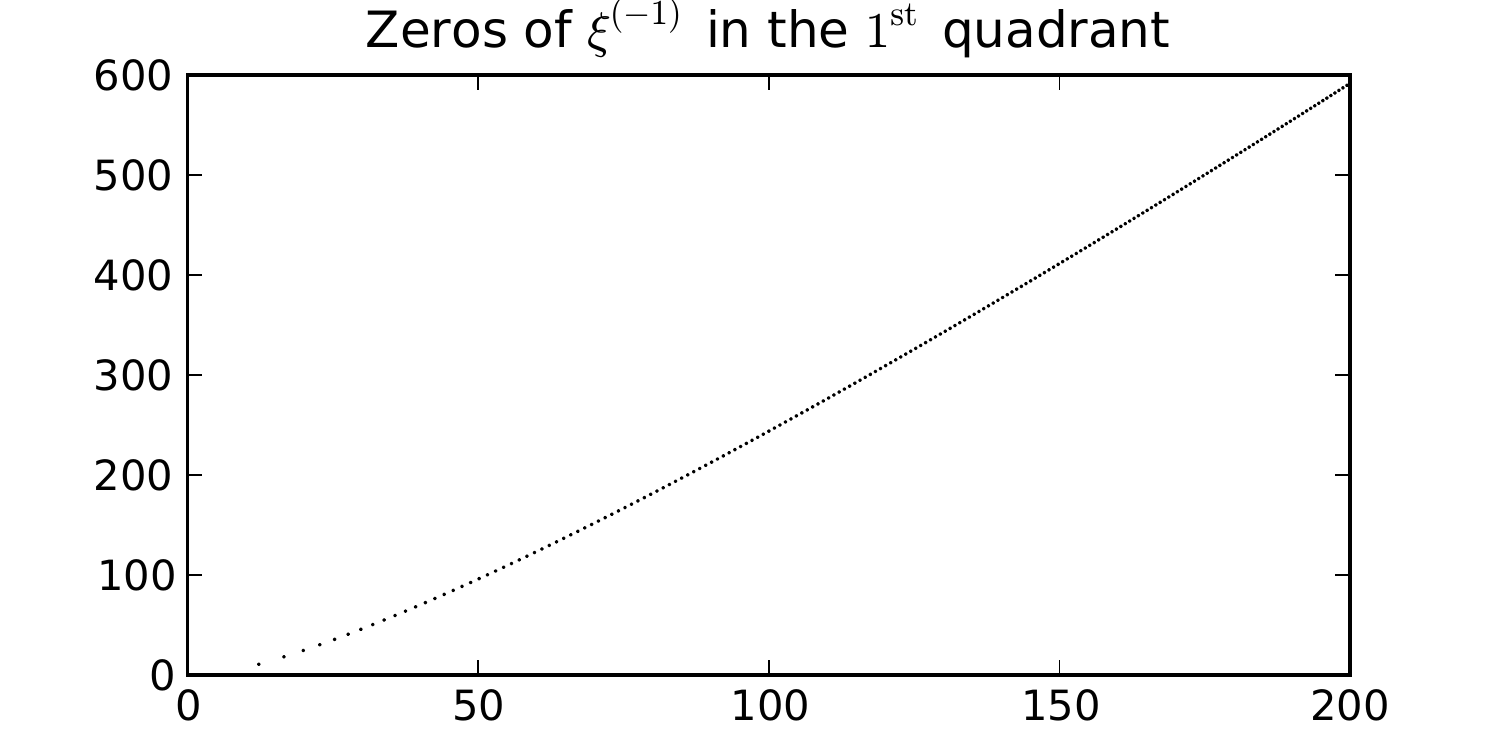}
\caption{
Plot of first $500$ zeros of $\xi^{(-1)}(s)$ in the first quadrant.
}
\label{fig33}
\end{figure}

Next we consider the distribution of zeros of $\xi^{(-1)}(s; -iA_0) :=\xi^{(-1)}(s) - i A_0$.
This function has
$$
\lim_{t \to \infty} \xi^{(-1)}(\frac{1}{2}+ it; -iA_0)=0.
$$
These are plotted in Figure \ref{fig34} to height $180$. This data hints that infinitely
many zeros lie on the critical line. Perhaps this will be  a positive proportion of all zeros.
However, as  the height increases more zeros seem
to go off the line and up to  height $500$ only about $1/3$ of
the zeros are on  the critical line. (Note that Theorem~\ref{th22}
shows that $\xi^{(-1)}(s; -iA_0)$ has only finitely
many zeros on the critical line in the lower half plane.)
%
%
%

\begin{figure}
\centering
\includegraphics[scale=1.00]{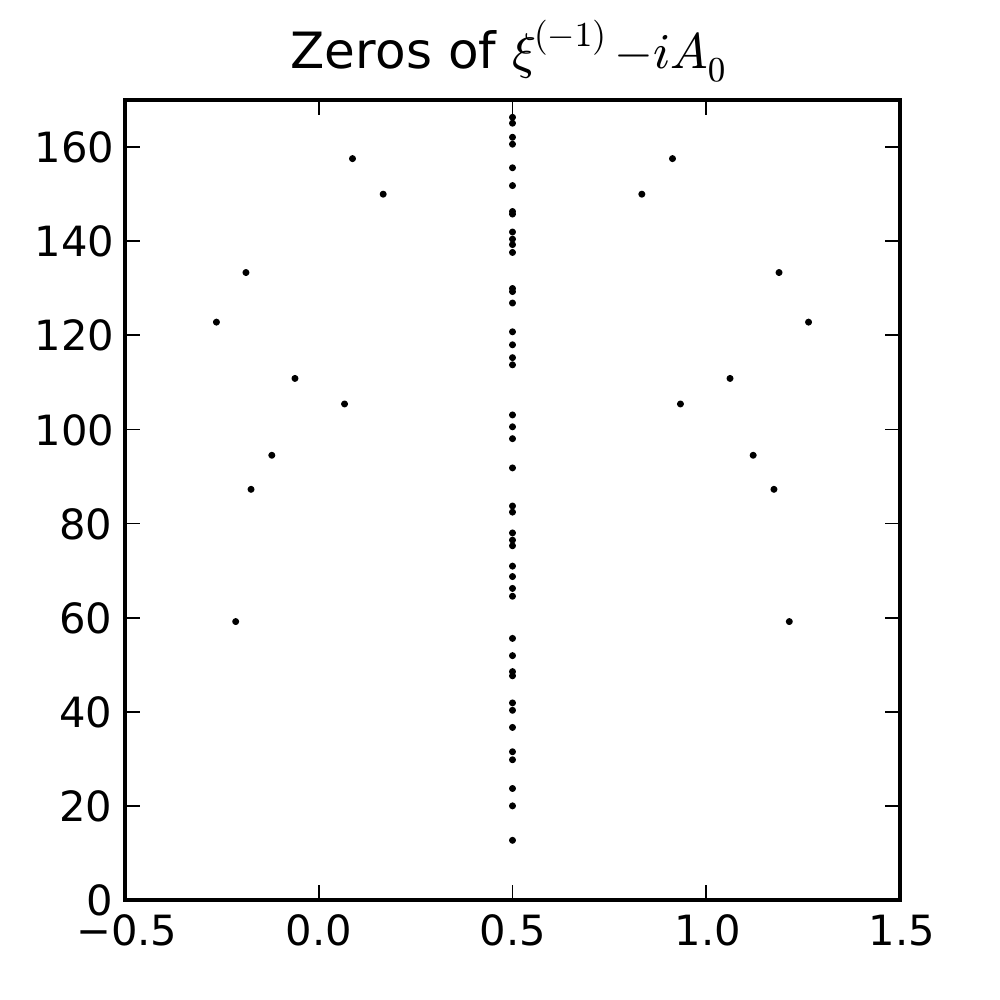}
\caption{
 Plot of zeros of $\xi^{(-1)}(s)- A_0 i$, where $A_0 \approx 2.80668$,
  to height $T=160$
}
\label{fig34}
\end{figure}

These computational results suggest a number of further questions.

{\em Question 1.} Let the zeros of $\xi^{(-1)}(s)$ have its zeros
$\rho= \sigma+it$ in the quadrant $\sigma\ge \frac{1}{2}, t\ge 0$
arranged  $\rho_n= \sigma_n + i t_n$
 arranged in order of increasing $|\rho_n|$
have the property that  $\sigma_n > \sigma_{n-1}$ and $t_n > t_{n-1}$?

The numerical evidence supports a positive answer to this question.
Furthermore, numerical  differencing of the abscissas and ordinates of the first 500
zeros uncovers regular trends. One may expect that there is 
an asymptotic expansion in functions  of $n$ for  the  spacings.

{\em Question 2.} What properties of the zeros of a
function like $\xi^{(-1)}(s)$, which do not lie on the
critical line,  would be sufficient  to  imply that the zeros
of its derivative  $\xi(s)$ would all lie on the critical line?

Would monotone increase of 
 the imaginary parts of the zeros in the first quadrant as the real part 
 increases, as in Question 1,  be a sufficient condition?  For an possibly related situation
 involving the $\xi$-function, where such a monotonicity
 implies the RH, see Haglund \cite{Hag11}.

{\em Question 3.} How would the GUE spacing distribution of zeros of $\xi(s)$ manifest iteself
in terms of the  distribution of the zeros of $\xi^{(-1)}(s)$?

Recall that the GUE hypothesis (see Odlyzko \cite{Odl87}, Berry and Keating \cite{BK99}, 
 Katz and Sarnak \cite{KS99}) asserts that when  zeros are ordered
by increasing ordinates, and zero spacings at
height $T$ are rescaled by a factor $\frac{1}{2\pi}\log T$ to have
expected spacing $1$,  then the distribution of spacings from height $ 0\le T \le X$
should as $X \to \infty$ approach a non-trivial continuous limiting distribution,
called the GUE distribution; this distribution arises
as an eigenvalue spacing distribution  in random matrix theory
for the Gaussian Unitary Ensemble.

The data above, in Table \ref{tab31} and
Figure \ref{fig33}, while rather limited, seems to suggest that 
the zero spacings of $\xi^{(-1)}(s)$ are extremely regular.
No fluctuations in spacings analogous to GUE seem visible
in this data. In contrast, fluctuations in zeta zero spacings are already
evident by $T=100$.

%

\begin{minipage}[t]{2.05in}
{Jeffrey C. Lagarias\\
Department of Mathematics\\
University of Michigan\\
Ann Arbor, MI 48109-1043, USA\\
\email{lagarias@umich.edu}}
\end{minipage}\bigskip

\begin{minipage}[t]{2.05in}
{David Montague\\
Department of Mathematics\\
University of Michigan\\
Ann Arbor, MI 48109-1043, USA}
\end{minipage}

\end{document}